\documentclass[12pt,reqno]{amsart}
\usepackage[utf8]{inputenc}

\usepackage{amsfonts, amssymb, amsmath, amsthm, mathtools, thmtools}
\usepackage{mathrsfs} 
\usepackage{latexsym}
\usepackage{enumerate}
\usepackage{multicol}
\usepackage{times}
\usepackage{verbatim}
\usepackage{tikz-cd}
\usepackage{fullpage}
\usepackage{here}
\usepackage{url}
\usepackage{csquotes}
\usepackage{placeins}
\usepackage{epic}
\usepackage{graphicx}
\usepackage{epstopdf}
\usepackage{pstricks}
\usepackage{pst-plot}
\usepackage{tikz}
\usepackage{xcolor}
\usepackage{subcaption}
\newcommand{\Chat}{\widehat{\C}}
\newcommand{\inv}{\iota}
\usetikzlibrary{shapes.geometric, positioning, calc}

\newenvironment{poc}{\begin{proof}[Proof of claim]}{\end{proof}}

\input xypic
\xyoption{all}
\xyoption{poly}

\usepackage[colorlinks=true]{hyperref}
\hypersetup{citecolor=blue, linkcolor=blue}

\usepackage[capitalise]{cleveref}

\crefname{equation}{}{}
\crefname{figure}{{\sc Figure}}{{\sc Figure}}
\crefname{subsection}{Subsection}{Subsections}

\usetikzlibrary{matrix,arrows,decorations.pathmorphing}

\usepackage[euler]{textgreek}
\usepackage{tikz,tikz-cd}
\usepackage[colorinlistoftodos, textwidth = 2.3cm]{todonotes}

\newtheorem{theorem}{Theorem}[section]

\newtheorem{lemma}[theorem]{Lemma}

\newtheorem{conjecture}[theorem]{Conjecture}
\newtheorem{claim}[theorem]{Claim}

\newtheorem*{claim*}{Claim}

\theoremstyle{definition}

\newcommand{\C}{{\mathbb C}}
\newcommand{\F}{{\mathbb F}}

\usepackage{tabstackengine}
\stackMath

\numberwithin{equation}{section} 
\numberwithin{figure}{section}
\numberwithin{table}{section}

\title{Multiplicative irreducibility of shifted multiplicative subgroups}
\author{Seoyoung Kim}
\address{Departement
Mathematik und Informatik, Universit\"at Basel, Spiegelgasse 1, 4051 Basel, Switzerland}
\email{seoyoung.kim@unibas.ch}
\author{Chi Hoi Yip}
\address{School of Mathematics\\ Georgia Institute of Technology\\ GA 30332\\ United States}
\email{cyip30@gatech.edu}
\author{Semin Yoo}
\address{Discrete Mathematics Group \\ Institute for Basic Science \\ 55 Expo-ro Yuseong-gu, Daejeon 34126 \\ South Korea}
\email{syoo19@ibs.re.kr}

\begin{document}

\begin{abstract}
In a recent breakthrough, Kalmynin resolved conjectures of Lev--Sonn and S\'{a}rk\"{o}zy on additive decompositions of multiplicative subgroups of prime fields. In this paper, inspired by a related conjecture of S\'{a}rk\"{o}zy, we prove multiplicative analogues of Kalmynin's results. We show that for every proper multiplicative subgroup \(G\), the shifted set \((G-1)\setminus\{0\}\) cannot be written as a product set nontrivially, addressing a conjecture of S\'{a}rk\"{o}zy. In addition, we prove that no nonzero shift of any coset of a proper multiplicative subgroup is a ratio set of the form \(A/A\). Our results substantially sharpen previous theorems of Shkredov and the authors.
\end{abstract}

\maketitle

\section{Introduction}

Throughout the paper, let $p$ be a prime, $\F_p$ the finite field with $p$ elements, and $\F_p^*=\F_p\setminus \{0\}$. For two subsets $A,B$ of $\F_p$, we define their \emph{sumset} $A+B=\{a+b: a \in A, \ b\in B\}$ and \emph{product set} $AB=\{ab:a\in A, \ b \in B\}$. Similarly, we can define their \emph{difference set} $A-B=\{a-b:a\in A, \ b\in B\}$  and \emph{ratio set} $A/B=\{a/b: a\in A, \ b\in B\}$ (provided that $0\notin B$). Following \cite{LS17, S12}, we use $\mathcal{R}_p$ to denote the set of nonzero squares in $\F_p$.

A central theme in arithmetic combinatorics and analytic number theory is the interplay between addition and multiplication. In particular, there is a substantial body of literature on additive and multiplicative decompositions of sets possessing various arithmetic structures. A celebrated conjecture in this direction is the so-called \emph{inverse Goldbach problem}, due to Ostmann \cite{Ostmann}, which states that the set of primes cannot be written as a nontrivial sumset (with finitely many exceptions allowed). The inverse Goldbach problem is still open, and we refer to the best-known progress by Elsholtz--Harper \cite{EH15} and Shao \cite{Shao16}. Similarly, Erd\H{o}s conjectured that any ``small perturbations" of the set of perfect squares cannot be written as a nontrivial sumset; we refer to the best-known progress due to S\'ark\"ozy and Szemer\'edi~\cite{SS65}. The multiplicative analogues of these two questions are also well-studied. Elsholtz \cite{E08} showed that the set of shifted primes cannot be written as a nontrivial product set (with finitely many exceptions allowed). More recently, Hajdu--S\'{a}rk\"{o}zy \cite{HS20} and the second author \cite{Y26} studied multiplicative decompositions of ``small perturbations" of the set of shifted $k$-th powers. 

In this paper, we study finite field analogues of these questions. In particular, we study certain multiplicative decompositions of shifted multiplicative subgroups of finite fields. To motivate these questions and our main results, we begin with additive versions of these questions. The following well-known conjecture is due to S\'{a}rk\"{o}zy \cite{S12}. 

\begin{conjecture}[{S\'{a}rk\"{o}zy \cite{S12}}]\label{conj:additiveSarkozy}
Let $p$ be a large enough prime. Then $\mathcal{R}_p$ admits no nontrivial additive decomposition, that is, there are no two subsets $A,B$ of $\F_p$ with $|A|,|B|\geq 2$, such that \[A+B=\mathcal{R}_p.\]
\end{conjecture}

It is natural to consider the analogue of this conjecture over multiplicative subgroups of finite fields. In particular, the following generalization of S\'{a}rk\"{o}zy's conjecture is widely believed:
\begin{conjecture}\label{conj:additiveSarkozy2}
Let $d\ge 2$ be fixed. Then for all sufficiently large prime powers $q\equiv 1 \pmod d$, 
the multiplicative subgroup $G$ of $\F_q$ of index $d$ admits no nontrivial additive decomposition.
\end{conjecture}

These two conjectures have been studied extensively; see for example
\cite{S13, Sh14, S16, LS17, S20, HP, Y24, Y25, Kalmynin}. In a breakthrough work of 2021, Hanson and Petridis \cite{HP} proved the following theorem using Stepanov's method \cite{S69}.
\begin{theorem}[Hanson and Petridis \cite{HP}]\label{thm:HP}
Let $p$ be a prime and $G$ be a proper multiplicative subgroup of $\F_p$. If $A,B$ are subsets of $\F_p$ such that $A+B \subseteq G \cup \{0\}$, then $$|A||B|\leq |G|+|(-A) \cap B|.$$ 
\end{theorem}
Together with Ford's theorem \cite{F08} on the distribution of divisors of shifted primes, Hanson and Petridis showed that Theorem \ref{thm:HP} implies an asymptotic version of \cref{conj:additiveSarkozy} (see \cite[Corollary 1.4]{HP}). 

In a very recent breakthrough, Kalmynin \cite{Kalmynin} used an ingenious idea to resolve Conjecture~\ref{conj:additiveSarkozy} and made substantial progress towards \cref{conj:additiveSarkozy2} over prime fields. More precisely, he proved the following theorem:
\begin{theorem}[Kalmynin \cite{Kalmynin}]\label{thm:Kalmynin}
Let $p$ be a prime and $G$ be a proper multiplicative subgroup of $\F_p$. If $A,B$ are subsets of $\F_p$ such that $A+B=G$ and $|A|,|B|\geq 2$, then $|G|$ is a perfect square and $|A|=|B|=\sqrt{|G|}$. Moreover, when $G=\mathcal{R}_p$, such sets $A,B$ do not exist.
\end{theorem}

To prove \cref{thm:Kalmynin}, he employed a sophisticated application of Stepanov's method, building on the auxiliary polynomial introduced by Hanson and Petridis~\cite{HP}. He also noted, in a talk at the Number Theory Web Seminar, that the techniques used to resolve \cref{conj:additiveSarkozy} do not appear to extend to general multiplicative subgroups.
\smallskip

There are also other variants of \cref{conj:additiveSarkozy} and \cref{conj:additiveSarkozy2}. For example, the restricted sumset variant of Conjecture~\ref{conj:additiveSarkozy} was confirmed by Shkredov \cite{Sh14} and the second author \cite{Y25} for all finite fields $\F_q$ with $q$ odd and $q>13$.

In the same paper, using similar techniques, Kalmynin \cite{Kalmynin} resolved the following difference set variant of Conjecture~\ref{conj:additiveSarkozy} proposed by Lev and Sonn \cite{LS17}. 

\begin{conjecture}[Lev and Sonn \cite{LS17}]\label{conj:LS}
Let $p$ be a large enough prime. Then there is no subset $A$ of $\F_p$ such that 
\[A-A=\mathcal{R}_p \cup \{0\}.\]
\end{conjecture}

Kalmynin \cite{Kalmynin} proved the following generalization of \cref{conj:LS} for multiplicative subgroups of prime fields. We note that the case $|G|\leq p^{4/5-o(1)}$ in the following theorem has been proved by Shkredov \cite{S16}.

\begin{theorem}[Kalmynin \cite{Kalmynin}]\label{thm:LS}
Let $p$ be a prime and let $G$ be a proper multiplicative subgroup of $\F_p$ such that $|G|\notin \{2,6\}$. Then for each subset $A$ of $\F_p$, we have \[A-A \ne G \cup \{0\}.\]
\end{theorem}

In fact, it is implicit in his proof \cite[Section 3]{Kalmynin} that a slightly stronger statement holds: under the same assumption of \cref{thm:LS}, if $A-A\subseteq G\cup\{0\}$, then
\begin{equation}\label{eq:d-1}
|A|^2 - |A|\le |G|-1. 
\end{equation}
Note that in the same setting, \cref{thm:HP} implies the slightly weaker bound that $|A|^2-|A|\leq |G|$. However, improving the upper bound from $|G|$ to $|G|-1$ requires highly nontrivial efforts. As another consequence, when $p\ge 17$ and $p \equiv 1 \pmod 4$, inequality~\cref{eq:d-1} slightly improves the well-known Hanson--Petridis bound $\frac{\sqrt{2p-1}+1}{2}$ on the clique number of the Paley graph over $\F_p$ \cite{HP} (that is, the largest possible size of $A \subseteq \F_p$ such that $A-A \subseteq \mathcal{R}_p \cup \{0\}$) to $\frac{\sqrt{2p-5}+1}{2}$. We refer to \cite{Y25a} for more discussions on recent progress towards estimating the clique number of Paley graphs and their generalizations.

\smallskip

Next, we turn to multiplicative analogues of the conjectures and results discussed above, before discussing our contributions. These analogues were also initiated by S\'{a}rk\"{o}zy \cite{S14}.

\begin{conjecture}[{S\'{a}rk\"{o}zy \cite{S14}}]\label{conj:M_S}
Let $p$ be a large enough prime. Then for each $\lambda\in \F_p^*$, the set $(\mathcal{R}_p-\lambda)\setminus\{0\}$ has no nontrivial multiplicative decomposition, that is,  there are no two subsets $A,B$ of $\F_p$ with $|A|,|B|\geq 2$, such that \[AB=(\mathcal{R}_p-\lambda)\setminus\{0\}.\]
\end{conjecture}

We note that the removal of $0$ is necessary, since if $0\in \mathcal{R}_p-\lambda$ then trivially we have the  decomposition
\[(\mathcal{R}_p-\lambda)=\{0,1\}\cdot (\mathcal{R}_p-\lambda).\] 
In \cite{KYY}, inspired by \cref{conj:additiveSarkozy2}, we formulated a generalization of \cref{conj:M_S} for proper multiplicative subgroups. 

\begin{conjecture}[{\cite[Conjecture 1.9]{KYY}}]\label{conj:MS2}
Let $d \geq 2$ be fixed. Let $q \equiv 1 \pmod d$ be a sufficiently large prime power. Let $G$ be the multiplicative subgroup of $\F_q$ of index $d$. Then 
for each $\lambda\in\F_q^*$, the set $(G-\lambda)\setminus\{0\}$
has no nontrivial multiplicative decomposition.
\end{conjecture}

Using Stepanov's method, we made some partial progress on \cref{conj:MS2} in our previous paper \cite[Theorem 1.11]{KYY}. The following lemma is a simplified version of \cite[Theorem 1.1]{KYY}, which is the key ingredient in the proof of \cite[Theorem 1.11]{KYY} in the same paper.

\begin{lemma}[{\cite[Theorem 1.1]{KYY}}]\label{thm: stepanovea}
Let $p$ be a prime and let $G$ be a proper multiplicative subgroup of $\F_p$. Let $A,B \subseteq \F_p^*$ and $\lambda \in \F_p^*$. 
If $AB+\lambda \subseteq G \cup \{0\}$, then $$|A||B| \leq |G|+|B \cap (-\lambda A^{-1})|+|A|-1.$$ 
Moreover, when $\lambda \in G$, we have a stronger upper bound:
$$|A||B| \leq |G|+|B \cap (-\lambda A^{-1})|-1.$$
\end{lemma}

In particular, for $\lambda=1$, we showed that \cref{thm: stepanovea} implies an asymptotic version of \cref{conj:M_S} in \cite[Theorem 1.11]{KYY}. We also showed that it implies weaker versions of \cref{conj:MS2} in \cite[Corollary 6.4 and Theorem 6.6]{KYY}: if $G$ is a proper multiplicative subgroup of a prime field $\F_p$ with $|G|$ sufficiently large, then for all $\lambda \in \F_p^*$, the set $(G-\lambda)\setminus \{0\}$ cannot be written as a product set $AA$ or $ABC$ nontrivially.
 
 We remark that the multiplicative S\'{a}rk\"{o}zy conjecture is potentially more difficult than the additive S\'{a}rk\"{o}zy conjecture \footnote{Private communication with Ilya Shkredov.}. Indeed, in Shkredov's paper \cite{S20}, he made important progress on \cref{conj:additiveSarkozy2} by showing that all small multiplicative subgroups of prime fields have no nontrivial additive decomposition. By contrast, he obtained only a much weaker statement toward multiplicative decompositions of shifted multiplicative subgroups: if $\epsilon>0$ is fixed, then for all \emph{small} multiplicative subgroups $G$ of $\F_p$ satisfying $1\ll_{\epsilon}|G|\ll_{\epsilon} p^{6/7-\epsilon}$, $$A/A\neq (\xi G+1)\setminus \{0\}$$ holds for all $\xi\in \F_p^*$ and subsets $A$ of $\F_p^*$, where $$\xi G=\{\xi g: g\in G\}.$$ Similarly, when $\lambda \notin G$, the conclusion of \cref{thm: stepanovea} is weaker and not strong enough for our intended application to \cref{conj:MS2}.

Our first main result resolves \cref{conj:MS2} for prime fields in the setting $\lambda\in G$.

\begin{theorem}\label{thm:sarkozy1}
Let $p$ be an odd prime, $G$ be a proper subgroup of $\F_p^*$, and $\lambda\in G$. Then there do not exist sets $A,B\subseteq \F_p^*$ with $|A|,|B|\ge 2$ such that
 \[
 AB=(G-\lambda)\setminus\{0\}.
 \]
\end{theorem}

\cref{thm:sarkozy1} can be viewed as a multiplicative analogue of \cref{thm:Kalmynin}: it is somewhat stronger in the sense that \cref{thm:Kalmynin} does not resolve the general subgroup setting; however, it is somewhat weaker in the sense that we need the additional assumption $\lambda\in G$. We believe that the case $\lambda \notin G$ is substantially more difficult, and our techniques do not seem to extend to the case $\lambda\in \F_p^*\setminus G$. We also remark that \cref{thm:sarkozy1} does not always hold when $\lambda\in \F_p^*\setminus G$. Indeed, a short computer search finds the following two counterexamples:
\begin{enumerate}
    \item In $\F_{11}$, we have $(G-\lambda)\setminus \{0\}=AB$ for
\[
G=\{1,3,4,5,9\}, \quad \lambda=2, \quad A=\{1,7\}, \quad B=\{1,2,3\}.
\]
\item In $\F_{19}$, we have $(G-\lambda)\setminus \{0\}=AB$ for
\[
G=\{1,7,8,11,12,18\},\quad \lambda=2,\quad A=\{1,9\}, \quad B=\{6,9,18\}.
\]
\end{enumerate}

We also note that the assumption that $G$ is proper in \cref{thm:sarkozy1} is necessary. Indeed, let $g$ be a primitive root in $\F_p$ and set $A=\{g^j: 0\leq j \leq \frac{p-3}{2}\}$; then \[A/A=\left\{g^j: -\frac{p-3}{2}\leq j\leq \frac{p-3}{2}\right\}=\F_p^*\setminus \{-1\}=(\F_p^*-1)\setminus \{0\}.\]

Next, we consider the simpler setting of expressing a shifted multiplicative subgroup as a ratio set $A/A$. Recall that Shkredov's result confirms a multiplicative analogue of the generalized Lev--Sonn conjecture for small multiplicative subgroups. Our second main result, stated below, completely establishes a multiplicative analogue of \cref{thm:LS}, and in particular extends the above result of Shkredov \cite{S20} to all proper multiplicative subgroups of order at least $3$. 

\begin{theorem}\label{thm: mul version of Lev-Sonn}
Let $p$ be an odd prime. Let $G$ be a proper multiplicative subgroup of $\F_p$ with $|G|\geq 3$, and let $\xi,\mu\in \F_p^*$. Then we have the following.
\begin{enumerate}
    \item The set of nonzero elements of $\xi G+\mu$ cannot be written as a ratio set of the form $A/A$: for any $A \subseteq \F_p^*$, we have
    \[
A/A \neq (\xi G+\mu)\setminus \{0\}. 
\]
\item The set of nonzero elements of $(\xi G \cup \{0\})+\mu$ cannot be written as a ratio set of the form $A/A$: for any $A \subseteq \F_p^*$, we have
\[
A/A \neq \big((\xi G \cup \{0\})+\mu\big)\setminus \{0\}.
\]
\end{enumerate}
\end{theorem}

We believe that both parts of the theorem are natural multiplicative analogues of \cref{thm:LS}. Also, we remark that the assumptions that $G$ is proper and $|G|\geq 3$ are necessary (in both parts of the above theorem) by considering the following counterexamples:
\begin{itemize}
    \item $|G|=1$: we have $G=\{1\}$ and we can take $A=\{1\}$ so that $A/A=(2G-1)\setminus \{0\}$ and $A/A=((-G\cup \{0\})+1)\setminus \{0\}$.
    \item $|G|=2$: we have $G=\{1,-1\}$ and we can take $A=\{1\}$ so that $A/A=(\frac{1}{2}G+\frac{1}{2})\setminus \{0\}$. By taking $A=\{1,-2\}$, we have $A/A=\{1,-1/2, -2\}=((-\frac{3}{2}G\cup \{0\})-\frac{1}{2})\setminus \{0\}$.
    \item $G=\F_p^*$: the remark following \cref{thm:sarkozy1} shows it is possible to have $A/A=(\F_p^*-1)\setminus \{0\}.$ By taking $A=\F_p^*$, we simply have $A/A=(\F_p+1)\setminus\{0\}$.
\end{itemize}

Our proofs of Theorems~\ref{thm:sarkozy1} and~\ref{thm: mul version of Lev-Sonn} rely on a combination of Stepanov's method and tools from symmetric polynomials. While they are motivated by Kalmynin's paper~\cite{Kalmynin}, the actual arguments are quite different, and the resulting proofs seem to be simpler. We also note that our methods apply more generally to small multiplicative subgroups of arbitrary finite fields. For example, with only minor modifications, the proofs of Theorems~\ref{thm:sarkozy1} and~\ref{thm: mul version of Lev-Sonn} extend to multiplicative subgroups $G$ of finite fields $\F_q$ satisfying $|G|\leq p-2\sqrt{p}$, where $p$ is the characteristic of $\F_q$.

\medskip

Finally, we consider the setting of roots of unity in the field of complex numbers. Let $\C$ be the field of complex numbers and $\C^*=\C\setminus \{0\}$. Let $S^1$ denote the unit circle, that is, $S^1=\{z\in \C: |z|=1\}$. In Kalmynin's paper \cite{Kalmynin}, as well as the paper of Hanson and Petridis \cite{HP}, they also considered analogues of Theorems~\ref{thm:HP} and~\ref{thm:LS} for finite subgroups of $S^1$. In a similar spirit, in the following theorem, we establish an analogue of Theorems~\ref{thm:sarkozy1} and~\ref{thm: mul version of Lev-Sonn} for finite subgroups of $S^1$. Note that in this case, the proof is much easier compared to the finite field setting, as we can apply a geometric argument.

\begin{theorem}\label{thm:complex}
Let $G$ be a finite subgroup of $S^1$.
\begin{enumerate}
\item Let $\lambda\in G$. Then there do not exist sets $A,B\subseteq \C^*$ with $|A|,|B|\ge 2$ such that
 \[
 AB=(G-\lambda)\setminus\{0\}.
 \]
    \item If $|G|\geq 3$, then for every $\xi,\mu\in\C^*$ and every $A\subseteq\C^*$,
\[
A/A\ \neq\ (\xi G+\mu)\setminus\{0\}.
\]
\item If $|G|\geq 3$, then for every $\xi,\mu\in\C^*$ and every $A\subseteq\C^*$,
\[
A/A \neq \big((\xi G \cup \{0\})+\mu\big)\setminus \{0\}.
\]
\end{enumerate} 
\end{theorem}
Similar to the remark following \cref{thm: mul version of Lev-Sonn}, the assumption $|G|\geq 3$ is also necessary in (2) and (3). 

\subsection*{Organization of the paper}
In \cref{sec:mHP}, we revisit a multiplicative analogue of the Hanson--Petridis polynomial introduced in our previous paper \cite{KYY}, which plays a crucial role in the proof of \cref{thm: stepanovea}. In \cref{sec:3}, we prove \cref{thm: mul version of Lev-Sonn}, establishing multiplicative analogues of the generalized Lev--Sonn conjecture. In \cref{sec:sec4}, we resolve multiplicative analogues of the generalized S\'{a}rk\"{o}zy conjecture in the case $\lambda\in G$ (\cref{thm:sarkozy1}). 
Finally, in \cref{sec4}, we present a short proof of \cref{thm:complex}.

\section{A multiplicative analogue of Hanson--Petridis polynomials}\label{sec:mHP}

In this section, we revisit a multiplicative analogue of the Hanson--Petridis polynomial developed in our previous paper \cite{KYY}, and derive useful properties from it. 

Let $G$ be a proper multiplicative subgroup of $\F_p$. Let $A,B \subseteq \F_p^*$ and $\lambda \in \F_p^*$ with $|A|, |B| \geq 2$, such that $AB+\lambda \subseteq G \cup \{0\}$. 
Denote $A=\{a_1,a_2,\ldots, a_n\}$ with $|A|=n$ and $B=\{b_1,b_2,\ldots,b_m\}$ with $|B|=m$.
Let $r:=|B\cap(-\lambda A^{-1})|$. Relabel the elements in $B$ so that
$b_1,\dots,b_r\in B\cap(-\lambda A^{-1})$ and $b_{r+1},\dots,b_m\in B\setminus(-\lambda A^{-1})$.

Following \cite[Section 4.1]{KYY}, there exists a unique solution $c_1,c_2,\ldots,c_n \in \F_p$ to the following system of equations: 
\begin{equation} \label{system} 
\left\{
\TABbinary\tabbedCenterstack[l]{
\sum_{i=1}^n c_i=1\\\\
\sum_{i=1}^n c_i a_i^j=0,  \quad 1 \leq j \leq n-1.
} \right.    
\end{equation}

\medskip

The following lemma will be a key observation in the proof of \cref{thm:sarkozy1}.

\begin{lemma}\label{lem:GF}
The coefficients $c_1,c_2,\dots,c_n$ satisfy
\[
\sum_{i=1}^n \frac{c_i a_i^n}{1-a_ix}
= (-1)^{n-1}\Big(\prod_{i=1}^n a_i\Big)\cdot\frac{1}{\prod_{i=1}^n(1-a_ix)}
\]
as an identity of rational functions in $\F_p(x)$.
\end{lemma}

\begin{proof}
First, we compute the coefficients $c_1,\dots,c_n$ explicitly in terms of the Vandermonde matrix. To do so, define
$$W = \begin{bmatrix}
  1 & a_1 & a_1^2 & \cdots & a_1^{n - 1} \\
  1 & a_2 & a_2^2 & \cdots & a_2^{n - 1} \\
\vdots & \vdots & \ddots & \vdots \\
  1 & a_n & a_n^2 & \cdots & a_n^{n - 1} \\
\end{bmatrix}.
$$
Then, $W$ is invertible and $W^t [c_1, c_2, \ldots, c_n]^t=[1,0,\ldots, 0]^t$, so \[[c_1,\ldots, c_n]^t=(W^{-1})^t [1,0,\ldots, 0]^t,\] which is precisely the first row of $W^{-1}$. Therefore, using the formula for the inverse of a Vandermonde matrix (see for example \cite[Section 0.9.11]{HJ13}), we obtain
\begin{equation}\label{eq:ciexplicit}
c_i=\frac{(-1)^{n-1}a_1a_2 \cdots a_n}{a_i \cdot \prod_{j\neq i}(a_i-a_j)}.   
\end{equation}

Using equation \cref{eq:ciexplicit}, we have
\begin{align}
\sum_{i=1}^n \frac{c_i a_i^n}{1-a_ix}
&=\sum_{i=1}^n 
\frac{(-1)^{n-1}\prod_{t=1}^n a_t}{a_i\prod_{j\neq i}(a_i-a_j)}\cdot \frac{a_i^n}{1-a_ix} \nonumber\\
&=(-1)^{n-1}\Big(\prod_{t=1}^n a_t\Big)\sum_{i=1}^n
\frac{a_i^{n-1}}{\prod_{j\neq i}(a_i-a_j)}\cdot \frac{1}{1-a_ix}. \label{eq:compare}
\end{align}
Now, write
\[
\frac{1}{\prod_{j=1}^n(1-a_jx)}
=\frac{1}{\prod_{j=1}^n a_j}\cdot \frac{1}{\prod_{j=1}^n(a_j^{-1}-x)}.
\]
By the partial fraction decomposition, we have
\begin{equation}\label{eq:PF}
\frac{1}{\prod_{j=1}^n(a_j^{-1}-x)}
=\sum_{i=1}^n \frac{1}{a_i^{-1}-x}\cdot
\frac{1}{\prod_{j\neq i}(a_j^{-1}-a_i^{-1})}.
\end{equation}
Note that for a fixed $i$, we have
\[
\prod_{j\neq i}(a_j^{-1}-a_i^{-1})
=\prod_{j \ne i}\frac{a_i-a_j}{a_ia_j}
=\frac{\prod_{j\neq i}(a_i-a_j)}{a_i^{n-1}\prod_{j\neq i}a_j},
\]
thus
\[
\frac{1}{\prod_{j\neq i}(a_j^{-1}-a_i^{-1})}
=\frac{a_i^{n-1}\prod_{j\neq i}a_j}{\prod_{j\neq i}(a_i-a_j)}.
\]
Substituting these identities into equation \eqref{eq:PF} yields
\[
\frac{1}{\prod_{j=1}^n(a_j^{-1}-x)}
=\Big(\prod_{j=1}^n a_j\Big)
\sum_{i=1}^n
\frac{a_i^{n-1}}{\prod_{j\neq i}(a_i-a_j)}\cdot \frac{1}{1-a_ix}.
\]
Multiplying both sides by $\big(\prod_{j=1}^n a_j\big)^{-1}$ gives
\[
\frac{1}{\prod_{j=1}^n(1-a_jx)}
=\sum_{i=1}^n
\frac{a_i^{n-1}}{\prod_{j\neq i}(a_i-a_j)}\cdot \frac{1}{1-a_ix}.
\]
Thus, comparing with equation \cref{eq:compare}, the lemma follows.
\end{proof}

Following \cite[Section 4.1]{KYY}, consider the following auxiliary polynomial
\begin{equation}
\label{auxiliary1}
f(x)=-\lambda^{n-1}+\sum_{i=1}^n c_i (a_ix+\lambda)^{n-1+|G|}\in \F_p[x].    
\end{equation}
We showed that $f$ is a non-zero polynomial, each of $b_1,\dots,b_r$ is a root of $f$ with multiplicity at least $n-1$,
and each of $b_{r+1},\dots,b_m$ is a root of $f$ with multiplicity at least $n$.
It follows that
\begin{equation}\label{eq:degf}
r(n-1)+(m-r)n=mn-r \le \deg f \le n-1+|G|.
\end{equation}
Thus, when equality holds in \cref{eq:degf}, $f$ must have degree $n-1+|G|=mn-r$ and 
\begin{equation}\label{eq:notSd}
f(x)=C \cdot \prod_{j=1}^r (x-b_j)^{n-1} \cdot \prod_{j=r+1}^m (x-b_j)^{n},    
\end{equation}
where
$$C=\sum_{i=1}^n c_i a_i^{n-1+|G|} \neq 0.$$

\medskip

Next, additionally assume that $AB+\lambda\subseteq G$ and $\lambda\in G$, which applies to the proof of \cref{thm:sarkozy1} in \cref{sec:sec4}. In this case, we have $r=0$.
Indeed, if $b \in B \cap (-\lambda A^{-1})$, then we have $ab+\lambda=0$ for some $a \in A$, contradicting $0 \notin G$.
Thus, by the above discussion, each of $b_1,\dots,b_m$ is a root of $f$ with multiplicity at least $n$. Also, observe that system~\cref{system} implies that
\[f(0)=-\lambda^{n-1}+\sum_{i=1}^n c_i \lambda^{n-1+|G|}=-\lambda^{n-1}+\lambda^{n-1}=0,\] 
as well as that the coefficient of $x^j$ of $f$ is $0$ for $1\leq j \leq n-1$. Thus, $0$ is also a root of $f$ with multiplicity at least $n$. By comparing the sum of multiplicities of the roots $0, b_1, b_2, \dots, b_m$ of $f$ and the degree of $f$, we have the following stronger inequality:
\begin{equation}\label{eq:degflambda}
(m+1)n\leq |G|+n-1.
\end{equation}
Moreover, when equality holds in \cref{eq:degflambda}, $mn=|G|-1$ and we have the following factorization of $f$:
\begin{equation}\label{eq:f-when-lambda-in-G}
f(x)=C\cdot x^n\prod_{j=1}^m (x-b_j)^n
   =C\Bigl(x\prod_{j=1}^m(x-b_j)\Bigr)^n.
\end{equation}
where
$$C=\sum_{i=1}^n c_i a_i^{n-1+|G|} \neq 0.$$

\section{Proof of \cref{thm: mul version of Lev-Sonn}}\label{sec:3}

In this section, we prove \cref{thm: mul version of Lev-Sonn}.

\subsection{Symmetric polynomials}\label{sec: 2.1}
We first recall some standard families of symmetric polynomials and fix our notation. 
Throughout, $n\ge 1$ is an integer and $x_1,\dots,x_n$ are indeterminates. 

\begin{itemize}
    \item The \textit{power sum symmetric polynomials} are defined by
    \[ p_k(x_1,x_2,\dots,x_n)
        = \sum_{i=1}^n x_i^k\]
    for $k\ge 0$.

    \item The \textit{elementary symmetric polynomials} are defined by
    \[ e_k(x_1,x_2,\dots,x_n)
        = \sum_{1\le i_1< i_2< \cdots <i_k\le n}
          x_{i_1}x_{i_2}\cdots x_{i_k} \]
    for $1\le k \le n$. We set $e_0(x_1,x_2,\dots,x_n)=1$ and $e_k(x_1,x_2,\dots,x_n)=0$ for $k>n$.
\end{itemize}

The following classical identity relates these two families of symmetric polynomials, which will be used in \cref{subsec:3.2}.
\begin{lemma}[Newton's identities \cite{Stanley}]\label{lem:newton}
For each integer $k\ge 1$, we have
\[
ke_{k}(x_{1},x_{2},\dots ,x_{n})=\sum_{i=1}^{k}(-1)^{i-1}e_{k-i}(x_{1},x_{2},\dots ,x_{n})p_{i}(x_{1},x_{2},\dots ,x_{n}),
\]
where the identity holds whenever $n\geq k \geq 1$. 
\end{lemma}

\subsection{Proof of \cref{thm: mul version of Lev-Sonn}(1)}
Suppose for contradiction that there exist $\xi,\mu\in \F_p^*$, a subset $A\subseteq\F_p^*$, and a proper multiplicative subgroup $G$ of $\F_p$ with order at least $3$, such that
$$A/A = (\xi G+\mu)\setminus\{0\}.$$ 

Let $|A|=n$. Since $|G|\geq 3$, we have $n\geq 2$.
Since $A/A = (\xi G+\mu)\setminus\{0\}$, we have $A/(\xi A)=(G+\mu/\xi)\setminus \{0\}$. Set $\lambda=-\mu/\xi$ so that $$A/(\xi A)+\lambda=G \setminus \{\lambda\},$$ and we follow the same notations as in \cref{sec:mHP} by viewing $B=1/(\xi A)$.

Observe first that $r=|(\xi A)^{-1}\cap (-\lambda A^{-1})|=0$.
Indeed, if we have $u \in (\xi A)^{-1}\cap (-\lambda A^{-1})$, then there exist $a,a' \in A$ such that $u=\xi^{-1}a^{-1}=-\lambda (a')^{-1}$. It then follows that \[a'/(\xi a)=-\lambda \in A/(\xi A),\]
and thus $0=-\lambda+\lambda \in A/(\xi A)+\lambda \subseteq G$, a contradiction.

Since $1/\xi=a/(\xi a)$ for all $a\in A$, as a trivial upper bound on $A/(\xi A)$, we have 
\begin{equation}\label{eq:ub}
|G \setminus \{\lambda\}|=|A/(\xi A)|\le n^2-n+1.    
\end{equation}

Next, we derive an upper bound on $n^2$. If $\lambda \in G$, \cref{thm: stepanovea} implies that 
$$
n^2=|A|^2\leq |G|+r-1=|G|-1,
$$
which violates inequality~\eqref{eq:ub}. Thus, we must have $\lambda \notin G$.
In this case, inequality~\eqref{eq:ub} implies that $|G|=|G\setminus\{\lambda\}|\leq n^2-n+1$. On the other hand, by \cref{thm: stepanovea}, we get 
\[n^2=|A|^2 \le |G|+r+|A|-1=|G|+n-1.\]
Thus, we have $n^2-n+1=|G|$, and equality holds in \cref{eq:degf}. Thus, the auxiliary polynomial $f$ defined in equation~\cref{auxiliary1} also satisfies equation~\cref{eq:notSd}. It follows that we have the following equality of polynomials:
\begin{equation}\label{eq:ff}
f(x)=-\lambda^{n-1}+\sum_{i=1}^n c_i (a_ix+\lambda)^{n^2}=C \cdot \prod_{j=1}^n \bigl(x-\xi^{-1}a_j^{-1}\bigr)^{n}, 
\end{equation}
where $C=\sum_{i=1}^n c_i a_i^{n-1+|G|}\neq 0$.

In the middle expression of equation~\cref{eq:ff}, for each $1\leq k \leq n^2$, the coefficient of $x^{k}$ is
\[\binom{n^2}{k}\lambda^{n^2-k}\sum_{i=1}^n c_i a_i^k.\]
By system \eqref{system}, $\sum_{i=1}^n c_i a_i^k=0$ for $1\le k\le n-1$.
Thus, the coefficient of $x^k$ in $f(x)$ is $0$ for $1\le k\le n-1$. 

As for the right-hand side of equation~\cref{eq:ff}, we set
\[
P(x):=\prod_{j=1}^n (1-\xi a_jx)=\sum_{t=0}^n (-\xi)^t e_t(a_1,\dots,a_n)x^t.
\]
Using $x-\xi^{-1}a_j^{-1}=-\xi^{-1}a_j^{-1}(1-\xi a_jx)$, we have
\[f(x)=C\prod_{j=1}^n(-1)^n(\xi^{-n}a_j^{-n})P(x)^n,\]
and thus the coefficient of $x^k$ in $f(x)$ is
\[
C(-\xi)^{-n^2}\Big(\prod_{j=1}^n a_j\Big)^{-n} \cdot [x^k]P(x)^n,
\]
where $[x^k]P(x)^n$ denotes the coefficient of $x^k$ in $P(x)^n$. Therefore, the coefficient of $x^k$ in $P(x)^n$ is $0$ for $1\le k\le n-1$.

\smallskip

We now prove by induction that $e_k:=e_k(a_1,\ldots,a_n)=0$ holds for all $1\le k\le n-1$.
For $k=1$, the coefficient of $x$ in $P(x)^n$ is $n(-\xi) e_1$, so we have $n(-\xi) e_1=0$.
Thus $e_1=0$.

Assume $e_1=\cdots=e_{k-1}=0$ for some $2\le k\le n-1$.
Then, we have
\[
P(x)=1+(-\xi)^k e_k x^k+x^{k+1}Q_k(x)
\]
for some polynomial $Q_k(x)\in \F_p[x]$. It follows that
\[
P(x)^n = 1 + n(-\xi)^k e_k x^k + x^{k+1}R_k(x).
\]
for some polynomial $R_k(x)\in \F_p[x]$. Since the coefficient of $x^k$ in $P(x)^n$ is $0$, we have $n(-\xi)^k e_k=0$, that is, $e_k=0$.

We have thus proved that $e_1=\cdots=e_{n-1}=0$.
This implies that $$\prod_{i=1}^n(x-a_i)=x^n-e_1x^{n-1}+e_2x^{n-2}-\cdots+(-1)^ne_n=x^n+(-1)^n e_n.$$ 
In particular, $a_1,a_2,\ldots, a_n$ are precisely the $n$ distinct roots of the polynomial $x^n-\alpha$ for $\alpha:=(-1)^{n-1}e_n\ne 0$. It follows that $A$ is a coset of the subgroup of $n$-th roots of unity in $\F_p^*$. Thus, we have $|A/(\xi A)|=n$.
But we also have
$|A/(\xi A)|=|G|=n^2-n+1$. Since $n\ge 2$, we have $n^2-n+1>n$, a contradiction.
This completes the proof. 
\mbox{}\hfill $\square$

\smallskip
\subsection{Proof of \cref{thm: mul version of Lev-Sonn}(2)}\label{subsec:3.2}

Suppose for contradiction that there exist $\xi,\mu\in\F_p^*$ and a set $A\subseteq\F_p^*$ such that
\begin{equation}\label{eq:assumption}
A/A = \bigl((\xi G \cup \{0\})+\mu\bigr)\setminus \{0\}.
\end{equation}
It follows that $|A/A|\ge|G|\geq 3$, and thus we have $n:=|A|\geq 2$. 

\begin{claim}\label{claim:ngeq3}
We have $n\geq 3$.
\end{claim}
\begin{poc}
Suppose otherwise that $n=2$. In this case, we must have $|A/A|=|G|=3$ and $0\in \xi G+\mu$. Without loss of generality, write $A=\{1,c\}$ with $c\in\F_p^*$. Then $A/A=\{1,c,c^{-1}\}$, and we must have $c\neq \pm 1$ so that $c \neq c^{-1}$. 

Since $-\mu\in \xi G$, we have $\xi G/(-\mu)=G$. Thus, by dividing $-\mu$ on both sides of equation~\cref{eq:assumption}, 
we obtain
\[
(A/A)/(-\mu)=\bigl((G\cup\{0\})-1\bigr)\setminus\{0\}.
\]
Let $G=\{1,\omega,\omega^2\}$ with $1+\omega+\omega^2=0$ and $\omega\neq 1$. Then we have
\[
\{1,c,c^{-1}\}=\{\mu, \mu(1-\omega),\mu(1-\omega^2)\}.
\]
Next, we derive a contradiction in each of the following three cases:
\begin{enumerate}
    \item $\mu=1$. In this case we have $\{1,c,c^{-1}\}=\{1,1-\omega, 1-\omega^2\}$. However, \[(1-\omega)(1-\omega^2)=1-\omega-\omega^2+1=2-\omega-\omega^2=3\neq 1.\] 
    \item $1=\mu(1-\omega)$. In this case we have $\{1,c,c^{-1}\}=\{\mu, 1,\mu(1-\omega^2)\}$. However, \[\mu^2(1-\omega^2)=\mu^2(1-\omega)(1+\omega)=\frac{1+\omega}{1-\omega}\neq 1.\] 
     \item $1=\mu(1-\omega^2)$. In this case we have $\{1,c,c^{-1}\}=\{\mu, \mu(1-\omega),1\}$. However, 
     \[
     (1-\omega)(1+\omega)^2=\omega-\omega^2\neq 1,
     \]
     and thus
     \[\mu^2(1-\omega)=\frac{\mu^2(1-\omega)}{\mu^2(1-\omega^2)^2}=\frac{1}{(1-\omega)(1+\omega)^2}\neq 1.\] 
\end{enumerate}
This shows that $n=2$ is impossible, as required.
\end{poc}

\begin{claim}\label{claim:neq1}
We have $\mu\notin \{1,-1\}$.
\end{claim}
\begin{poc}\

(1) Suppose otherwise that $\mu=1$. Then we have $A/A-1 \subseteq \xi G \cup \{0\}$. 
Without loss of generality, we may assume that $1\in A$. Then for each $a\in A\setminus \{1\}$, we have 
\[
a-1\in \xi G, \qquad \frac{1}{a}-1=\frac{1-a}{a} \in \xi G,
\]
thus $-1/a\in G$, that is, $-a\in G$. 

By \cref{claim:ngeq3}, we have $n\geq 3$. Let $a,b\in A \setminus \{1\}$ with $a\neq b$. Then we have
$$
\frac{a}{b}-1=\frac{a-b}{b}\in \xi G, \qquad \frac{b}{a}-1=\frac{b-a}{a}\in \xi G,
$$
and their ratio lies in $G$ as follows
\[
\frac{(a/b)-1}{(b/a)-1}
=\frac{(a-b)/b}{(b-a)/a}
=-\frac{a}{b}\in G.
\]
On the other hand, since $-a,-b\in G$, we have $\frac{a}{b}=\frac{-a}{-b}\in G$. This, together with $-a/b \in G$, implies $-1\in G$. 

The above discussion shows that $-A \subseteq G$ and thus $A/A\subseteq G$. On the other hand, by equation~\cref{eq:assumption}, we have $|A/A|\geq |G|$. It follows that
$$
G=\bigl((\xi G \cup \{0\})+1\bigr)\setminus \{0\}
$$
and so $-1\in \xi G$. Since $-1\in G$, we have $\xi G=G$ and thus
$$
G=\bigl((G \cup \{0\})+1\bigr)\setminus \{0\}.
$$
This implies that for any $g \in G$ with $g \ne 1$, we also have $g-1\in G$.
Since $-1\in G$, inductively we have $-2,-3,\ldots,-(p-1)\in G$. Thus, $G=\F_p^*$, violating the assumption that $G$ is proper.

(2) Suppose otherwise that $\mu=-1$. We have $A/A+1 \subseteq \xi G \cup \{0\}$. By \cref{claim:ngeq3}, we have $n\geq 3$. Without loss of generality, we may assume that $1\in A$. Then for each $a\in A\setminus \{-1\}$, we have 
\[
a+1\in \xi G, \qquad \frac{1}{a}+1=\frac{a+1}{a}\in \xi G,
\]
and taking the ratio of these two elements of $\xi G$, $a$ lies in $G$. Thus, $A \setminus \{-1\} \subseteq G$. 

Now, we want to show $A\subseteq G$ (and this leads to a contradiction). Suppose otherwise that $-1\in A$ and $-1\notin G$. Since $n\geq 3$, we can pick $a\in A \setminus \{1,-1\}$. Then we have 
$$
\frac{a}{-1}+1=1-a\in \xi G, \qquad \frac{-1}{a}+1=\frac{a-1}{a} \in \xi G.
$$
Since $a\in G$, it follows that $1-a, a-1\in \xi G$ and thus $-1\in G$, a contradiction. 

Thus, we have shown that $A \subseteq G$ and thus $A/A \subseteq G$. Now, a similar argument as in the proof of (1) shows that $G=\F_p^*$, again violating the assumption.
\end{poc}

Divide both sides of equation \cref{eq:assumption} by $\xi$ and set \[\lambda:=-\frac{\mu}{\xi}.\] 
Then, we have
\begin{equation}\label{eq:size0}
A/(\xi A)=\bigl((G\cup\{0\})-\lambda\bigr)\setminus\{0\}.
\end{equation}

Let $r=|(\xi A)^{-1} \cap (-\lambda A^{-1})|$. Then, \[r=|\{(a,b)\in A\times A \colon a/(b\xi)=-\lambda\}|\] and thus $-\lambda$ has exactly $r$ representations in $A/(\xi A)$. Note that if $a,b\in A$ such that $a/(\xi b)=-\lambda$, then $b/(\xi a)=-1/(\lambda \xi^2)$. Thus, $-1/(\lambda \xi^2)$ has exactly $r$ representations in $A/(\xi A)$ as well. Also, $1/\xi$ has exactly $n$ representations in $A/(\xi A)$. Note that $-\lambda=\mu/\xi$ and $-1/(\lambda \xi^2)=1/(\mu \xi)$. We have $\mu\neq \pm 1$ by \cref{claim:neq1}, thus $-\lambda, -1/(\lambda \xi^2), 1/\xi$ are pairwise distinct. 

Suppose first that $\lambda\in G$. It then follows from equation~\cref{eq:size0} that
\begin{equation}\label{eq:ubG}
|G|= |(G \cup \{0\}) \setminus \{\lambda\}|=|A/(\xi A)+\lambda|\leq n^2-(n-1)-2(r-1)=n^2-n-2r+3.
\end{equation}
Now, by \cref{thm: stepanovea} and inequality~\cref{eq:ubG}, we get 
\[n^2=|A|^2 \le |G|+|(\xi A)^{-1}\cap (-\lambda A^{-1})|-1=|G|+r-1\leq n^2-n-r+2,\]
a contradiction since $n \ge 3$.

Therefore, $\lambda \notin G$. Then equation~\eqref{eq:size0} becomes
\begin{equation}\label{eq:size}
A/(\xi A)=(G\cup\{0\})-\lambda.    
\end{equation}
We follow the same notations as in \cref{sec:mHP} by viewing $B=1/(\xi A)$.

Note that $r\geq 1$ since $-\lambda \in A/(\xi A)$. By equation~\cref{eq:size}, we have
\begin{equation}\label{eq:ubG1}
|G|+1=|G \cup \{0\}|=|A/(\xi A)+\lambda|\leq n^2-n+1-2(r-1)=n^2-n-2r+3.
\end{equation}
Moreover, the equality in~\cref{eq:ubG1} holds precisely when each ratio in $A/(\xi A)$ other than $1/\xi,-\lambda,$ and $-(\lambda \xi^2)^{-1}$ has exactly one representation.
On the other hand, by \cref{thm: stepanovea} and inequality~\cref{eq:ubG1}, we get 
\[n^2=|A|^2 \le |G|+|(\xi A)^{-1}\cap (-\lambda  A^{-1})|+|A|-1=|G|+r+n-1 \leq n^2-r+1.\]
Thus, we must have $r=1$ and $n^2-n=|G|$. Moreover, we deduce that each ratio in $A/(\xi A)$ other than $1/\xi$ has exactly one representation.

Since $n^2-n=|G|$, the equality holds in \cref{eq:degf}. Thus, the auxiliary polynomial $f$ defined in equation~\cref{auxiliary1} also satisfies equation~\cref{eq:notSd}. It follows that we have the following equality of polynomials:
\begin{equation}\label{eq:f1111}
f(x)=-\lambda^{n-1}+\sum_{i=1}^n c_i (a_ix+\lambda)^{n^2-1}=C\bigl(x-(\xi a_1)^{-1}\bigr)^{n-1}\prod_{j=2}^n \bigl(x-(\xi a_j)^{-1}\bigr)^n,
\end{equation}
where $C\neq 0$, and by definition, $a_1$ is the unique element in $A$ such that \[(\xi a_1)^{-1} \in (\xi A)^{-1} \cap (-\lambda A^{-1}),\] that is, $\mu a_1\in A$. By replacing the set $A$ with $A/a_1$ and repeating the above argument, we may, without loss of generality, assume that $a_1=1$. Then we must have $\mu \in A$. 

Recall that system \eqref{system} states that $\sum_{i=1}^n c_i=1$ and $\sum_{i=1}^n c_i a_i^k=0$ for $1\le k\le n-1$. Since $|G|=n^2-n$ and $\lambda\notin G$, it follows that 
$$
f(0)=-\lambda^{n-1}+\lambda^{n^2-1}\sum_{i=1}^n c_i=-\lambda^{n-1}+\lambda^{n^2-1}=\lambda^{n-1}(\lambda^{n^2-n}-1)\neq 0.
$$
On the other hand, a similar computation shows that the coefficient of $x^k$ in $f(x)$ is $0$ for $1\le k\le n-1$. Thus, there exists $R(x)\in\F_p[x]$ such that
$f(x)=f(0)+x^nR(x)$.
It follows that the coefficient of $x^k$ in $f'(x)$ is $0$ for $0\le k\le n-2$. On the other hand, differentiating \eqref{eq:f1111} and dividing it by $f$, we obtain
\[
\frac{f'(x)}{f(x)}=\frac{n-1}{x-\xi^{-1}}+\sum_{j=2}^n \frac{n}{x-(\xi a_j)^{-1}}.
\]
Using the identity
\[
\frac{1}{1-z}=1+z+z^2+\cdots+z^{n-2}+\frac{z^{n-1}}{1-z},
\]
we have
\begin{equation}\label{eq:log}
\frac{f'(x)}{f(x)}
=-\sum_{k=0}^{n-2}\Bigl((n-1)\xi^{k+1}+n\sum_{j=2}^n (\xi a_j)^{k+1}\Bigr)x^k
+x^{n-1}\cdot U(x),
\end{equation}
for some rational function $U(x)$ such that $U(x)f(x)$ is a polynomial.

Multiplying both sides of equation~\cref{eq:log} by $f(x)=f(0)+x^nR(x)$ and using the facts that $f'(x)$ is divisible by $x^{n-1}$ and $f(0)\neq 0$,
we obtain that for each $0\le k\le n-2$,
\[
(n-1)\xi^{k+1}+n\sum_{j=2}^n (\xi a_j)^{k+1}=0.
\]
Equivalently, for $1\le k\le n-1$,
\[
\sum_{j=2}^n a_j^{k}=-\frac{n-1}{n}.
\]
It follows that $p_k(a_2,a_3,\ldots, a_n)$ is uniquely determined for $1\leq k\leq n-1$. Newton's identities (\cref{lem:newton}) then allow us to compute $e_k(a_2,a_3,\ldots, a_n)$ for $1\leq k \leq n-1$, and thus the set $\{a_2,a_3,\ldots, a_n\}$ is uniquely determined, say it is given by $A^*$. 

As a summary, we have shown the following claim.
\begin{claim}
If $\{1,\mu\}\subseteq A$ and $A$ satisfies equation~\cref{eq:size}, then $A=\{1\} \cup A^*$, where $A^*$ is the set defined above, and $\mu\in A^*$. Moreover, each ratio in $A/(\xi A)$ other than $1/\xi$ has exactly one representation.
\end{claim}

Now, set $\widetilde{A}=\mu/A$.
Note that we still have
$$
\frac{\widetilde{A}}{\xi \widetilde{A}}=\frac{A}{\xi A} = (G \cup \{0\})-\lambda.
$$
Moreover, since $1,\mu \in A$, we also have $1,\mu\in \widetilde{A}$. 
Thus, we can apply the above claim to $\widetilde{A}$ to conclude that $\widetilde{A}=\{1\} \cup A^*=A$. 

Since $n\geq 3$, we can take $a'\in A \setminus \{1,\mu\}$. Then, we have $\mu/a'\in A$ and $\mu/a'\neq 1$. 
The two ordered pairs
$(1,a')$ and $\Bigl(\frac{\mu}{a'},\mu\Bigr)$ are distinct and both lie in $A\times A$.
They give the same ratio in $A/(\xi A)$ as
\[
\frac{1}{\xi a'}=\frac{\mu/a'}{\xi\mu}.
\]
This shows that the ratio $1/(\xi a')$ (which is not $1/\xi$) has at least two representations in $A/(\xi A)$, a contradiction.
\mbox{}\hfill $\square$

\smallskip
\section{Proof of \cref{thm:sarkozy1}}\label{sec:sec4}

In this section, we prove \cref{thm:sarkozy1}. Throughout, we follow the notations in \cref{sec:mHP}. 

Let $G$ be a multiplicative subgroup of $\F_p^*$ and let $\lambda\in G$. Suppose for contradiction that there exist sets $A,B\subseteq \F_p^*$ with $|A|=n\ge2$, $|B|=m\ge2$ such that
\[
AB=(G-\lambda)\setminus\{0\}.
\]
Then we have $mn\geq |AB|=|G|-1$. Then equality holds in \cref{eq:degflambda}. Thus, $mn=|G|-1$ and
in particular $|AB|=|A||B|$. Thus, $G\setminus \{\lambda\}$ has a disjoint partition as 
\begin{equation}\label{eq:partition}
G \setminus \{\lambda\}=AB+\lambda = \bigsqcup_{a \in A} (\lambda + aB).    
\end{equation}
Moreover, by the discussion in \cref{sec:mHP}, 
we have the following factorization of the auxiliary polynomial $f$ from equation~\eqref{eq:f-when-lambda-in-G}:
\begin{equation}\label{eq:lambda-in-G}
f(x)=C\Bigl(x\prod_{j=1}^m(x-b_j)\Bigr)^n.
\end{equation}
where
$$C=\sum_{i=1}^n c_i a_i^{n-1+|G|} \neq 0.$$

Next, we prove the following claim.

\begin{claim}\label{lem:residue}
For every $b\in B$, we have the identity
\[
\frac{|G|-1}{n+1}\cdot
\frac{\displaystyle\sum_{i=1}^n \frac{c_i a_i^{n+1}}{(a_ib+\lambda)^2}}
{\displaystyle\sum_{i=1}^n \frac{c_i a_i^{n}}{(a_ib+\lambda)}}
=
n\left(\frac{1}{b}+\sum_{\substack{b'\in B\\ b'\neq b}}\frac{1}{b-b'}\right).
\]
\end{claim}

\begin{poc}
Fix $b\in B$. By equation \cref{eq:lambda-in-G}, we can write
$f(x)=(x-b)^n h_b(x)$ with $h_b(b)\neq0$, where  
\[
h_b(x)=Cx^n \prod_{b'\neq b, b'\in B}(x-b')^n.
\]
Note that we have
\begin{equation}\label{eq:hb2}
\frac{h_b'(b)}{h_b(b)} = n\left(\frac{1}{b}+\sum_{\substack{b'\in B\\ b'\neq b}}\frac{1}{b-b'}\right).
\end{equation}

On the other hand, we can rewrite $h_b'(b)/h_b(b)$ in terms of derivatives of $f$ at $b$.
Since $f(x)=(x-b)^n h_b(x)$, we have
\[f^{(n)}(b)=n!h_b(b) \qquad \text{and} \qquad f^{(n+1)}(b)=(n+1)!h_b'(b).\]
It follows that
\begin{equation}\label{eq:der-ratio}
\frac{h_b'(b)}{h_b(b)}=\frac{f^{(n+1)}(b)}{(n+1)f^{(n)}(b)}.
\end{equation}

Finally, we compute $f^{(n)}(b)$ and $f^{(n+1)}(b)$ based on the definition of $f$ in equation~\cref{auxiliary1}. Let $N=|G|+n-1$.
For $1\leq k\leq N$, we have
\[
f^{(k)}(x)=(N)_k\sum_{i=1}^n c_i a_i^{k}(a_ix+\lambda)^{N-k},
\]
where \[(N)_k=N(N-1)\cdots(N-k+1).\]
Since $ab\in (G-\lambda)\setminus \{0\}$ for all $a\in A,b\in B$, we have $ab+\lambda\in G$ and $(a_ib+\lambda)^{|G|}=1$. Thus, for $1\leq i \leq n$, we have
\begin{align*}
(a_ib+\lambda)^{N-n}&=(a_ib+\lambda)^{|G|-1}=(a_ib+\lambda)^{-1},\\
(a_ib+\lambda)^{N-(n+1)}&=(a_ib+\lambda)^{|G|-2}=(a_ib+\lambda)^{-2}.
\end{align*}
Therefore, we have
\[
f^{(n)}(b)=(N)_n\sum_{i=1}^n \frac{c_i a_i^{n}}{a_ib+\lambda}
\quad \text{and} \quad
f^{(n+1)}(b)=(N)_{n+1}\sum_{i=1}^n \frac{c_i a_i^{n+1}}{(a_ib+\lambda)^2}.
\]
Substituting into equation \cref{eq:der-ratio} gives
\[
\frac{h_b'(b)}{h_b(b)}
= \frac{(N)_{n+1}}{(n+1)(N)_n}\cdot
\frac{\displaystyle\sum_{i=1}^n \frac{c_ia_i^{n+1}}{(a_ib+\lambda)^2}}
{\displaystyle\sum_{i=1}^n \frac{c_ia_i^{n}}{(a_ib+\lambda)}}.
\]
Since $N-n=|G|-1$, 
\[
\frac{(N)_{n+1}}{(n+1)(N)_n}=\frac{N-n}{n+1}=\frac{|G|-1}{n+1}.
\]
Combining with equation \cref{eq:hb2} completes the proof. 
\end{poc}

For each $b\in B$, denote
\[
H(b):=\frac{1}{b}+\sum_{\substack{b'\in B\\ b'\neq b}}\frac{1}{b-b'}.
\]
\cref{lem:residue} gives, for each $b\in B$, 
\begin{equation}\label{eq:relationX}
\frac{|G|-1}{n+1}\cdot
\frac{\displaystyle\sum_{i=1}^n \frac{c_ia_i^{n+1}}{(a_ib+\lambda)^2}}
{\displaystyle\sum_{i=1}^n \frac{c_ia_i^{n}}{(a_ib+\lambda)}}
= nH(b).
\end{equation}

Consider the polynomial
\[
P(x)=\prod_{i=1}^n (1+a_i x) \in \F_p[x].
\]
Observe that for each $b\in B$, since $a_ib\neq -\lambda$ for all $i$, we have $P(b/\lambda)\neq 0$.

By \cref{lem:GF}, we have
\begin{equation}\label{eq:P(x)}
\sum_{i=1}^n \frac{c_i a_i^n}{1+a_i x}
= (-1)^{n-1}\Big(\prod_{i=1}^n a_i\Big)\cdot \frac{1}{P(x)}.
\end{equation}
Differentiating equation~\cref{eq:P(x)} with respect to $x$ yields
\begin{equation}\label{eq:P'(x)}
\sum_{i=1}^n \frac{c_i a_i^{n+1}}{(1+a_i x)^2}
= (-1)^{n-1}\Big(\prod_{i=1}^n a_i\Big)\cdot \frac{P'(x)}{P(x)^2}.
\end{equation}
For each $b\in B$, by setting $x=b/\lambda$ in equations~\cref{eq:P(x)} and~\cref{eq:P'(x)} and taking their ratio, we have
\[
\frac{\displaystyle\sum_{i=1}^n  \frac{c_ia_i^{n+1}}{(a_ib+\lambda)^2}}
{\displaystyle\sum_{i=1}^n  \frac{c_ia_i^{n}}{(a_ib+\lambda)}}
= \frac{1}{\lambda}\cdot \frac{P'(b/\lambda)}{P(b/\lambda)}.
\]
Thus, substituting into equation \cref{eq:relationX}, for every $b\in B$, 
\[
H(b)=\frac{|G|-1}{n(n+1)}\cdot \frac{P'(b/\lambda)}{\lambda P(b/\lambda)}.
\]
Multiplying the above equation by $b$ and summing over $b\in B$, we obtain
\begin{equation}\label{eq:bHb1}
\sum_{b\in B} bH(b)
= \frac{|G|-1}{n(n+1)}\sum_{b\in B} \frac{bP'(b/\lambda)}{\lambda P(b/\lambda)}.
\end{equation}
On the other hand, observe that
\[
|B|(|B|-1)=\sum_{\substack{b,b'\in B\\ b\neq b'}}1=\sum_{\substack{b,b'\in B\\ b\neq b'}}\left(\frac{b}{b-b'}+\frac{b'}{b'-b}\right)=2\sum_{b\in B}\ \sum_{\substack{b'\in B\\ b'\neq b}}\frac{b}{b-b'}.
\]
It follows from the definition of $H(b)$ that
\begin{equation}\label{eq:bHb2}
\sum_{b\in B} bH(b)=\binom{m+1}{2}.
\end{equation}
Thus, combining equations~\eqref{eq:bHb1} and~\eqref{eq:bHb2}, we have
\begin{equation}\label{eq:sum-log}
\sum_{b\in B} \frac{bP'(b/\lambda)}{\lambda P(b/\lambda)}
= \frac{n(n+1)}{|G|-1}\binom{m+1}{2}.
\end{equation}

Since \[\frac{P'(x)}{P(x)}=\sum_{i=1}^n \frac{a_i}{1+a_i x},\] 
for each $b\in B$, we have
\[
\frac{bP'(b/\lambda)}{\lambda P(b/\lambda)}=\sum_{i=1}^n \frac{a_i b}{\lambda+a_i b}
=\sum_{i=1}^n\left(1-\frac{\lambda }{\lambda +a_i b}\right)=n-\sum_{i=1}^n \frac{\lambda }{\lambda +a_i b}.
\]
Summing over $b\in B$ gives
\begin{equation}\label{eq:expand-sum}
\sum_{b\in B} \frac{bP'(b/\lambda )}{\lambda P(b/\lambda)}
=nm-\sum_{i=1}^n\ \sum_{b\in B}\frac{\lambda}{a_i b + \lambda}.
\end{equation}

Since the sets $\lambda+a_iB$ for $i=1,\dots,n$
form a disjoint partition of $G\setminus\{\lambda\}$ by equation~\eqref{eq:partition}, 
\[
\sum_{i=1}^n\sum_{b\in B}\frac{\lambda}{a_ib+\lambda}
=\lambda\sum_{g\in G\setminus\{\lambda\}}\frac{1}{g}.
\]
Also, since $|G|=mn+1\ge 5$, we have $\sum_{g\in G}g^{-1}=\sum_{g\in G}g=0$, and thus
\[
\sum_{g\in G\setminus\{\lambda\}}\frac{1}{g}=-\frac{1}{\lambda}.
\]
Therefore, the double sum equals $-1$, and equation \cref{eq:expand-sum} becomes
\begin{equation}\label{eq:sum-log=G}
\sum_{b\in B} \frac{bP'(b/\lambda)}{\lambda P(b/\lambda)}=nm+1=|G|.
\end{equation}
Plugging equation \cref{eq:sum-log=G} into equation \cref{eq:sum-log} gives
\[
|G|=\frac{n(n+1)}{|G|-1}\binom{m+1}{2}.
\]
Using $|G|-1=nm$, we obtain
\[
nm+1=\frac{(n+1)(m+1)}{2},
\]
so we have $(n-1)(m-1)=0$.
This contradicts the assumption $n,m\ge 2$ and completes the proof.

\section{The complex setting}\label{sec4}
In this section, we prove \cref{thm:complex}. 

We begin with a quick proof of \cref{thm:complex}(1).
\begin{proof}[Proof of \cref{thm:complex}(1)]
Suppose otherwise that there exist $A,B\subseteq \C^*$ with $|A|,|B|\ge 2$ such that
 $ AB=(G-\lambda)\setminus\{0\}.$ We may assume $\lambda=1$ without loss of generality by replacing $B$ with $B/\lambda$.  Let $|G|=n$. 

Fix $a\in A$. Since $ab+1\in G$ for every $b\in B$, we have $|ab+1|=1$, so $B$ is contained in the circle
\(
C_a:=\{z\in\C:\ |az+1|=1\}.
\)
If $a_1\neq a_2$ are two elements of $A$, then $C_{a_1}\neq C_{a_2}$, and hence \(B\subset C_{a_1}\cap C_{a_2}.\)
Since two distinct circles meet in at most two points, it follows that $|B|\le 2$. By symmetry, also $|A|\le 2$. Thus, we must have $|A|=|B|=2.$  Write $A=\{a_1,a_2\}$ and $B=\{b_1,b_2\}$.  

Let \(\zeta=e^{2\pi i/n}\), and write \(x_k:=\zeta^k-1\) for \(1\le k\le n-1\). Next, we prove the following claim. 

\begin{claim}\label{claim:x_k}
If $x_kx_\ell=x_mx_r$ for some \(1\le k,\ell,m,r\le n-1\), then $\{k,\ell\}=\{m,r\}.$
\end{claim}

\begin{poc}
For \(1\le t\le n-1\), we have
\[
x_t=\zeta^t-1
=e^{2\pi i t/n}-1
=2i\,e^{\pi i t/n}\sin\frac{\pi t}{n}.
\]
Hence
\[
x_kx_\ell
=-4\,e^{\pi i (k+\ell)/n}\sin\frac{\pi k}{n}\sin\frac{\pi \ell}{n}.
\]
So \(x_kx_\ell=x_mx_r\) implies
\[e^{\pi i (k+\ell-m-r)/n}=1 \qquad \text{and} \qquad \sin\frac{\pi k}{n}\sin\frac{\pi \ell}{n}=
\sin\frac{\pi m}{n}\sin\frac{\pi r}{n}.
\]
In particular, the first equation forces $k+\ell=m+r.$ Then the second equation above reduces to 
\[
\cos\frac{\pi(k-\ell)}{n}= \cos\frac{\pi(m-r)}{n},
\]
that is, $k-\ell=\pm(m-r).$ Together with \(k+\ell=m+r\), this gives
$\{k,\ell\}=\{m,r\}.$
\end{poc}

For $i,j\in \{1,2\}$, let \(k_{ij}\in\{1,\dots,n-1\}\) such that $a_ib_j=x_{k_{ij}}$. 
Then
\[
x_{k_{11}}x_{k_{22}}
=(a_1b_1)(a_2b_2)
=(a_1b_2)(a_2b_1)
=x_{k_{12}}x_{k_{21}}.
\]
\cref{claim:x_k} then implies that $
\{k_{11},k_{22}\}=\{k_{12},k_{21}\}.$
Thus either \(k_{11}=k_{12}\) or \(k_{11}=k_{21}\). Clearly, both cases are impossible.
\end{proof}

The proof of \cref{thm:complex} (2)(3) is slightly more subtle. To achieve that, we first review some basic facts on M\"obius transformations. Let $\Chat=\C\cup\{\infty\}$ be the Riemann sphere.
A \emph{M\"obius transformation} is a map
\[
M:\Chat\to\Chat,\qquad M(z)=\frac{az+b}{cz+d},
\]
where $a,b,c,d\in\C$ and $ad-bc\neq 0$.
We interpret the value at $\infty$ by
\[
M(\infty)=
\begin{cases}
a/c,& c\neq 0,\\
\infty,& c=0,
\end{cases}
\qquad\text{and}\qquad
M\!\left(-\frac{d}{c}\right)=\infty\ \ (c\neq 0).
\]
Every M\"obius transformation is bijective on $\Chat$, and its inverse map is also a M\"obius transformation. 
In particular, the inversion map $$\iota(z):=1/z$$ is a M\"obius transformation.
We will use the following standard facts (see, for example,
\cite[Ch.~3]{A78}):
\begin{itemize}
    \item M\"obius transformations send generalized circles (that is, circles or lines in $\Chat$) to generalized circles.
    \item M\"obius transformations are uniquely determined by their values on three distinct points.
    \item Two distinct generalized circles meet in at most two points.
\end{itemize}

The following lemma will be useful in the proof of \cref{thm:complex}.

\begin{lemma}
\label{lem:dihedral}
Let $m\ge 3$ and let $G=\{\omega^k:0\le k<m\}\subseteq S^1$ be the group of $m$-th roots of unity, where $\omega=e^{2\pi i/m}$.
If a M\"obius transformation $\psi$ satisfies $\psi(S^1)=S^1$ and $\psi(G)=G$, then there exists $\zeta\in G$ such that 
$\psi(z)=\zeta z$ for all $z\in \Chat$ or $\psi(z)=\zeta/z$ for all $z\in \Chat$.
\end{lemma}

\begin{proof}
Note that $\psi(S^1)=S^1$ and $\psi$ is continuous on $S^1$. Since $\psi$ permutes the elements in the set $G$, it has to preserve or reverse the cyclic order on the $m$ vertices of the regular $m$-gon $G$. Thus, $\psi$ restricted to $G$
is either a rotation $z\mapsto \zeta z$ or a reflection $z\mapsto \zeta/z$ for some $\zeta\in G$.
Since M\"obius maps are determined by their values on three distinct points, $\psi$ must agree with that rotation/reflection on all of $\Chat$.
\end{proof}

We end the paper with a proof of \cref{thm:complex} (2)(3).

\begin{proof}[Proof of \cref{thm:complex}(2)(3)]
Write $m:=|G|\ge 3$. Let $\xi,\mu\in\C^*$, and set
\[
V:=\mu+\xi G,\qquad C:=\mu+\xi S^1,
\]
where $C$ is the Euclidean circle of center $\mu$ and radius $|\xi|$. Let $T(z):=(z-\mu)/\xi$ for all $z\in \Chat$, so that $T(C)=S^1$ and $T(V)=G$. Define \[\psi:=T\circ\inv\circ T^{-1}.\] 

\medskip
(2) Suppose for contradiction that there exists  $A\subseteq\C^*$ such that
\[
A/A=X:=(\mu+\xi G)\setminus\{0\}.
\]
Since $A/A$ is closed under the inversion map $\inv: x \mapsto x^{-1}$, we have $\inv(X)=X$ and $1\in X$. Next, we consider the following two cases.

\medskip\noindent
\textbf{Case 1: $0\notin V$}. 
Then $X=V$ and thus $\inv(V)=V$.
Since $V$ contains $m\ge 3$ points of the circle $C$, the generalized circles $C$ and $\inv(C)$ meet in at least three points. It follows that $\inv(C)=C$ since two distinct generalized circles meet in at most two points, and thus $\psi(S^1)=S^1$. Moreover, for each $g\in G$, we have $v:=T^{-1}(g)=\mu+\xi g\in V$, hence $\inv(v)\in V$, so
\[
\psi(g)=T(\inv(v))\in T(V)=G.
\]
This, together with the injectivity of $\psi$, shows that $\psi(G)=G$. By Lemma~\ref{lem:dihedral}, we have  
$\psi(\infty)\in\{\infty,0\}$. However, a direct computation gives
\begin{equation}\label{eq:infty}
\psi(\infty)=T\bigl(\inv(T^{-1}(\infty))\bigr)=T(\inv(\infty))=T(0)=-\mu/\xi\in \C^*,
\end{equation}
a contradiction.

\medskip\noindent
\textbf{Case 2: $0\in V$}.
Then $X=V\setminus\{0\}$ has size $m-1$.

We first claim $m\geq 4$. Assume otherwise $m=3$. Then $|X|=2$. Since $\inv(X)=X$, it follows that $X=\{-1,1\}$. Write $G=\{1,\omega,\omega^2\}$ with $\omega=e^{2\pi i/3}$. Then we must have
\[
\{-1,0, 1\}=\{\mu+\xi, \mu+\xi \omega, \mu+\xi \omega^2\}.
\]
It follows that $0=(\mu+\xi)+(\mu+\xi \omega)+(\mu+\xi \omega^2)=3\mu$, that is, $\mu=0$, a contradiction. 

Let $g_0\in G$ be the unique element with $\mu+\xi g_0=0$. Since $0\in V \subseteq C$, the image $\inv(C)$ is a line. It follows that $\psi(S^1)$ is a line. For each $g\in G\setminus\{g_0\}$, we have $\mu+\xi g\in X$, so $\inv(\mu+\xi g)\in X$ and
\[
\psi(g)=T\bigl(\inv(\mu+\xi g)\bigr)\in T(X)=G\setminus\{g_0\}\subseteq S^1.
\]
Thus, since $\psi$ is injective, $\psi(G\setminus\{g_0\})$ is a set of $m-1\ge 3$ distinct points contained in the intersection of the line $\psi(S^1)$ with $S^1$,
contradicting the fact that a line meets a circle in at most two points. 

\medskip

(3) Suppose for contradiction that there exists  $A\subseteq\C^*$ such that
\[
A/A=X:=\big((\xi G \cup \{0\})+\mu\big)\setminus \{0\}.
\]
Since $A/A$ is closed under inversion, we have $\inv(X)=X$ and $1\in X$.

We first claim that $m \geq 4$. Suppose otherwise that $m=3$. Without loss of generality, assume that $1\in A$. Note that we can find an element $r\in A\setminus \{1,-1\}$ for otherwise $A/A \subseteq \{1,-1\}$. It follows that $A/A=\{1,r,r^{-1}\}$. We can then follow the proof of \cref{claim:ngeq3} verbatim to get a contradiction.
 
Next, we consider the following two cases.

\medskip\noindent
\textbf{Case 1: $0\notin V$}.
Then $X=V \cup \{\mu\}$. For each $v\in V$ we have $\inv(v)\in X$. Note that at most one $v\in V$ satisfies $\inv(v)=\mu$ (namely $v=1/\mu$), so for at least $m-1\ge 3$ distinct vertices $v\in V$,
we have $\inv(v)\in V\subseteq C$.
Hence, the generalized circles $C$ and $\inv(C)$ meet in at least three points. It follows that $\inv(C)=C$ and thus $\psi(S^1)=S^1$.
Moreover, since $0\notin V$, we have $T(X)=G\cup\{0\}$. For each $g\in G$,
$\mu+\xi g\in X$ implies $\psi(g)=T\bigl(\inv(\mu+\xi g)\bigr)\in T(X)=G\cup \{ 0\}$, and because $\psi(g)\in S^1$ we get $\psi(g)\in G$.
Thus $\psi(G)=G$.
By Lemma~\ref{lem:dihedral}, $\psi(\infty)\in\{\infty,0\}$.
However, as in equation~\eqref{eq:infty}, we have $\psi(\infty)=-\mu/\xi \in \C^*$, a contradiction. 

\medskip\noindent
\textbf{Case 2: $0\in V$}. Let $g_0\in G$ be the unique element with $\mu+\xi g_0=0$. 

First, we claim $m \geq 5$. Assume otherwise that $G=\{1,-1,i,-i\}$. Note that for $g\in G$, we have $\mu+\xi g=\mu-\mu g_0^{-1}g=\mu(1-g_0^{-1}g)$. Thus, $\mu+\xi G=\mu(1-G)$. It follows that
\[
X=\mu\{1,2,1-i,1+i\}, \qquad \inv(X)=\mu^{-1}\bigg\{1,\frac12,\frac{1+i}{2},\frac{1-i}{2}\bigg\}=\frac{1}{2\mu^2} X.
\]
Since $X=\inv(X)$, we must have $2\mu^2=1$. 
On the other hand, $1\in X$ implies that $\mu \in \{1, \frac{1}{2}, \frac{1+i}{2}, \frac{1-i}{2}\}$, a contradiction.

Since $0\in V$, we have $0\in C$ and thus $\inv(C)$ is a line. It follows that $\psi(S^1)$ is a line. One checks that $X=(V\setminus \{0\}) \cup \{\mu\}$, $T(V\setminus \{0\})=G\setminus\{g_0\}$, and $T(\mu)=0$, thus
\[
T(X)=\bigl(G\setminus\{g_0\}\bigr)\ \cup\ \{0\}.
\]
For each $g\in G\setminus\{g_0\}$, we have $\psi(g)\in \psi(S^1)$; on the other hand, $\mu+\xi g\in X$, hence $\inv(\mu+\xi g)\in X$ and 
\[\psi(g)=T\bigl(\inv(\mu+\xi g)\bigr)\in T(X) \subseteq S^1 \cup \{0\}.\]
Since $\psi$ is injective, it follows that there are $m-2\ge 3$ distinct points contained in the intersection of a line with $S^1$,
a contradiction. 
\end{proof}

\section*{Acknowledgments}
S. Kim and C.H. Yip thank Institute for Basic Science for hospitality during their visit, where part of this project was discussed. 
C.H. Yip also thanks Swaroop Hegde, Giorgis Petridis, and Ilya Shkredov for helpful discussions. C.H. Yip was supported in part by an NSERC fellowship. 
S. Yoo was supported by the Institute for Basic Science (IBS-R029-C1). 

\bibliographystyle{abbrv}
\bibliography{references}

\end{document}